\documentclass{amsart}

\usepackage{tikz,enumerate}

\usepackage{geometry}                
\usepackage{graphicx}
\usepackage{amssymb}
\usepackage{float}
\usepackage[colorlinks=true]{hyperref}
\hypersetup{urlcolor=blue, citecolor=red}
 
\usepackage{bbm}
\usepackage{amsmath,amscd}
\usepackage{algorithm,algorithmic}
\usepackage[english]{babel}
\usepackage[latin1]{inputenc}
\usepackage{times}
\usepackage[T1]{fontenc}
\usepackage{latexsym}
\usepackage{graphics}
\usepackage{color}
\usepackage{cite}
\usepackage{mathrsfs}
\usepackage{listings}

\usepackage{paralist}

\newtheorem{theorem}{Theorem}[section]

\newtheorem{lemma}[theorem]{Lemma}
\newtheorem{proposition}{Proposition}
\newtheorem{conjecture}{Conjecture}

\theoremstyle{definition}
\newtheorem{definition}[theorem]{Definition}
\newtheorem{remark}{Remark}
\newtheorem{example}{Example}
\newtheorem*{fact}{Fact}

\newcommand{\cH}{\mathcal{H}}

\newcommand{\cE}{\mathcal{E}}

\newcommand{\cN}{\mathcal{N}}

\newcommand{\cB}{\mathcal{B}}

\newcommand{\cG}{\mathcal{G}}
\newcommand{\cP}{\mathcal{P}}

\newcommand{\cK}{\mathcal{K}}

\newcommand{\bN}{\mathbb{N}}
\newcommand{\Z}{\mathbb{Z}}

\newcommand{\R}{\mathbb{R}}

\newcommand{\ep}{\epsilon }

\newcommand{\de}{\delta }

\newcommand{\ga}{\gamma }
\newcommand{\Ga}{\Gamma }

\newcommand{\ones}{\mathbbm{1}}

\newcommand{\MC}{\operatorname{MC}}
\newcommand{\ASFE}{\operatorname{ASFE}}

\newcommand{\Proj}{\operatorname{Proj}}

\newcommand{\bd}[1]{\partial #1}

\newcommand{\Mod}{\operatorname{Mod}}

\newcommand{\defeq}{\mathrel{\mathop:}=}
\newcommand{\cCeff}{\mathop{\mathcal{C}_{\eff}}}
\newcommand{\cReff}{\mathop{\mathcal{R}_{\eff}}}

\newcommand{\eff}{\textrm{eff}}
\newcommand{\Adm}{\operatorname{Adm}}

\newcommand{\bi}{\begin{itemize}}
\newcommand{\ei}{\end{itemize}}

\title[Modulus metrics]{Modulus metrics on networks}

\author[Nathan Albin and Nethali Fernando and Pietro Poggi-Corradini]{}

\thanks{The authors are supported by NSF  n.~1515810}

\email{albin@math.ksu.edu}
\email{tnethali@math.ksu.edu}
\email{pietro@math.ksu.edu}

\subjclass{90C35}

\date{}

\keywords{$p$-modulus, graph metrics, snowflaking, mincut, shortest path, effective resistance}

\begin{document}
\maketitle

\centerline{\scshape Nathan Albin}
\medskip
{\footnotesize
 \centerline{Department of Mathematics}
   \centerline{Kansas State University}
   \centerline{Manhattan, KS 66506, USA}
} 

\medskip

\centerline{\scshape Nethali Fernando}
\medskip
{\footnotesize
 \centerline{Department of Mathematics}
   \centerline{Kansas State University}
   \centerline{Manhattan, KS 66506, USA}
} 

\medskip

\centerline{\scshape Pietro Poggi-Corradini}
\medskip
{\footnotesize
 \centerline{Department of Mathematics}
   \centerline{Kansas State University}
   \centerline{Manhattan, KS 66506, USA}
} 

\medskip

\bigskip

\begin{abstract}
The concept of $p$-modulus gives a way to measure  the richness of a family of objects on a graph. In this paper, we investigate the families of connecting walks between two fixed nodes and show how to use $p$-modulus to form a parametrized family of graph metrics that generalize several well-known and widely-used metrics. We also investigate a characteristic of metrics called the "antisnowflaking exponent"  and present some numerical findings supporting a conjecture about the new metrics. We end with explicit computations of the new metrics on some selected graphs.
\end{abstract}

\section{Introduction}\label{sec:intro}

Throughout this paper, $G=(V,E)$ is a simple, finite, undirected and
connected network with nodes $V$ and edges $E$.  Our purpose in what
follows is to use a concept called $p$-modulus to derive a
parametrized family of metrics $d_p:V\times V\to\mathbb{R}$, to
interpret these metrics in the context of other well-known metrics on
graphs, and to analyze the behavior of the metrics as the parameter
$p$ varies.  To this end, we recall two fundamental definitions.
\begin{definition}\label{def:metric}
Let $X$ be a set and $d:X\times X\to\mathbb{R}$.  The function $d$ is called a \emph{metric} on $X$ if it satisfies the following three properties.
\begin{itemize}
\item[(i)] {\bf Non-negativity:} $d(a,b)\ge 0$ for all $a,b\in X$.
\item[(ii)] {\bf Non-degeneracy:} $d(a,b)=0$ if and only if $a=b$.
\item[(iii)] {\bf Symmetry:} $d(a,b)=d(b,a)$ for all $a,b\in X$.
\item[(iv)] {\bf Triangle inequality:} For every $a,b,c\in X$:
\[
d(a,b)\le d(a,c)+d(c,b).
\]
\end{itemize}
\end{definition}

\begin{definition}
  If, instead of (iv) in the previous definition, $d$ satisfies \begin{equation}\label{eq:ultrametric}
    \text{(iv)' }\,\, d(a,b)\leq \max\{d(a,c),d(c,b)\}, \qquad \text{for every $a,b,c\in X$,}
  \end{equation}
  then $d$ is called an \emph{ultrametric.}  Since (iv)' implies (iv),
  every ultrametric is a metric.
\end{definition}

When $d$ is a metric on $V$ (the set of nodes of a network), $d$ is
often referred to as a \emph{graph metric} or \emph{network metric}.
Three well-known network metrics are shortest path, effective
resistance and the (reciprocal of) minimum cut.  

The shortest path
metric between two nodes $a$ and $b$, as its name suggests, simply
refers to the length of the shortest path from $a$ to $b$.  The proof
that this quantity is a network metric is straightforward.

The effective resistance metric arises from viewing the graph $G$ as
an electrical circuit with unit resistances on each edge. The
effective resistance $\cReff(a,b)$ is the voltage drop necessary to
pass 1 amp of current between $a$ and $b$ through the network $G$
(see, e.g.,~\cite{doyle-snell1984}). Effective resistance also turns
out to be a metric on $V$, see for instance \cite[Corollary
10.8]{peres2009}. A practical way to compute $\cReff(a,b)$ is via the pseudo-inverse of the Laplacian. This is a $|V|\times|V|$ matrix $\cG$ with the property that 
\[
\cG L=L\cG=\Proj_{\langle 1\rangle^\perp}
\]
where $L$ is the combinatorial Laplacian 
\[
L:=\sum_{\{x,y\}\in E}(\de_y-\de_x)(\de_y-\de_x)^T
\]
and $\de_x$ is the indicator function of node $x$. With these notations
\[
\cReff(a,b):=(\de_b-\de_a)^T\cG(\de_b-\de_a).
\]

In order to define the minimum cut metric, we recall that a subset
$S\subset V$ is called an \emph{$ab$-cut} if $a\in S$ and
$b\not\in S$.  The size of a cut is measured by $|\bd S|$, where
$\bd S=\{e=\{x,y\}\in E: x\in S, y\not\in S\}$ is the {\it
  edge-boundary} of $S$.  In this paper, we shall use the notation
\begin{equation*}
  \MC(a,b) = \min\left\{|\bd S|:S\text{ is an $ab$-cut}\right\}\quad\text{and}\quad
  d_{\MC}(a,b) =
  \begin{cases}
    0 & \text{if }a =b,\\
    \MC(a,b)^{-1} & \text{if } a\ne b.
  \end{cases}
\end{equation*}
That $d_{\MC}$ is a graph metric (indeed, an ultrametric) can be seen
from the following argument.  Suppose $a$, $b$ and $c$ are distinct
vertices and let $S$ be a minimum cut for $\MC(a,b)$, so that
$a \in S$ and $b\not\in S$.  Then, either $c \in S$ or $c \not\in S$.
If $c \in S$, then $S$ is a $cb$-cut and $\MC(c,b)\le \MC (a,b)$,
hence $\MC(c,b)^{-1}\ge \MC (a,b)^{-1}$.  If $c \not\in S$, then $S$
is an $ac$-cut and $\MC(a,c)\le \MC (a,b)$, so that
$\MC(a,c)^{-1}\ge \MC (a,b)^{-1}$.  Since one of the two inequalities
must hold, it follows that
\begin{equation}\label{eq:ultrametric}
\MC(a,b)^{-1} \le \max \{\MC(a,c)^{-1},\MC(c,b)^{-1}\},
\end{equation}
showing that $d_{\MC}$ is an ultrametric.

A number of other interesting metrics exist on networks.  For example,
\cite{GASSP} presents a metric related to the spreading of epidemics
in a contact network.  There, the standard SI model of infection is
applied to a network with the spreading time from infected to
susceptible nodes modeled by independent exponential random variables.
In this system, the time required for an infection originating at node
$a$ to reach node $b$ is a random variable.  Its expected value is
called the Epidemic Hitting Time ${\rm EHT}(a,b)$ and was shown to be
a network metric.

In this paper we explore a new family of metrics arising from
$p$-modulus. The notion of $p$-modulus is a way to measure the
richness of families of walks (or other more general objects) in a
network.  In Section~\ref{sec:modulus} we review the basic theory of
$p$-modulus on graphs, recalling that $p$-modulus generalizes the
concepts of shortest path, effective resistance, and minimum cut
described above.  Then, in Section \ref{sec:dp}, we introduce a new
family of metrics, the $d_p$ metrics, which are obtained from the
$p$-modulus. In Section \ref{sec:as}, we also investigate a
characteristic of metrics called the ``antisnowflaking exponent,''
make a conjecture about the value of this exponent, and present some
numerical results to support this conjecture.  We end the paper by
calculating the $d_p$ metrics on some selected graphs and presenting
some problems we hope to answer in the future.

\section{Modulus on Networks}\label{sec:modulus}
\subsection{Definition of modulus}
A general framework for modulus of objects on networks was developed
in~\cite{apc}.  In what follows, $G= (V,E)$ is taken to be a finite
graph with vertex set $V$ and edge set $E$. We also assume for
simplicity that $G$ is undirected and simple.  The theory
in~\cite{apc} applies to any finite family of ``objects'' $\Ga$ for
which each $\ga\in\Ga$ can be assigned an associated function
$\cN(\ga,\cdot): E\rightarrow \R_{\ge 0}$ that measures the {\it usage
  of edge $e$ by $\ga$}.  Notationally, it is convenient to consider
$\cN(\ga,\cdot)$ as a row vector in $\R_{\ge 0}^E$.  For the purposes
of the present paper, it is sufficient to restrict attention to
families of walks.  A walk $\ga=x_0\ e_1\ x_1\ \cdots\ e_n\ x_n$ is a
string of alternating vertices and edges so that
$\{x_{k-1},x_{k}\}=e_k$ for $k=1,\dots, n$. To each such walk $\ga$ we
can associate the traversal-counting function $\cN(\ga,e):=$ number
times $\ga$ traverses $e$. So, in this case
$\cN(\ga,\cdot)\in\Z_{\ge 0}^E$.  In fact, if $\Gamma$ is the set of
all walks between two distinct vertices, then it turns out that
modulus can be computed by considering only simple paths (walks that
do not visit any node more than once), see \cite{ASGPC}.

We define a \emph{density} on
$G$ to be a nonnegative function on the edge set:
$\rho:E\to[0,\infty)$.  The value $\rho(e)$ can be thought of as the
{\it cost of using edge $e$}.   For an object $\ga\in\Ga$, we define 
\[
\ell_\rho(\ga):=\sum_{e\in E} \cN(\ga,e)\rho(e) = (\cN \rho)(\ga),
\]
which represents the {\it total usage cost} for $\ga$ with the given edge
costs $\rho$.  A density $\rho\in\R_{\ge 0}^E$ is 
{\it admissible for $\Ga$}, if
\[
\ell_\rho(\Ga)\defeq\inf_{\ga\in\Ga}\ell_\rho(\ga) \geq 1.
\]
Let
\begin{equation}\label{eq:Adm}
\Adm(\Ga):=\left\{\rho\in\R_{\ge 0}^E: \ell_\rho(\Ga)\geq  1\right\}
\end{equation}
be the set of admissible densities.

Given an exponent $p\in [1,\infty]$ we define the \emph{$p$-energy} of
a density $\rho$ as
\begin{equation*}
  \cE_{p}(\rho) \defeq \sum_{e\in E} \rho(e)^p\quad\text{if }p<\infty\qquad\mbox{and}\qquad
  \cE_{\infty}(\rho) \defeq
  \lim_{p\to\infty}\left(\cE_{p}(\rho)\right)^{\frac{1}{p}} =
  \max_{e\in E}\rho(e).
\end{equation*}
\begin{definition}\label{def:mod}
  Let $G=(V,E)$ be a simple finite graph and let $\Ga$ be a finite
  non-trivial family of objects with usage matrix
  $\cN\in\mathbb{R}^{\Ga\times E}$.  For $p\in [1,\infty]$, the {\it
    $p$-modulus} of $\Gamma$ is
\[ \Mod_{p}(\Ga)\defeq \inf_{\rho\in \Adm(\Ga)}\cE_{p}(\rho)\]
\end{definition}

\begin{remark}\label{rem:remarks}

  \begin{itemize}
  \item[(a)] When $\rho=\rho_0\equiv 1$, we frequently drop the
    subscript in $\ell_{\rho}$; $\ell(\ga)\defeq\ell_{\rho_0}(\ga)$
    simply counts the number of hops taken by walk $\ga$.
  \item[(b)] If $\Ga\subset\Ga'$, then
    $\Adm(\Ga')\subset \Adm(\Ga)$, so
    $\Mod_{p}(\Ga)\leq \Mod_{p}(\Ga')$, for all
    $1\leq p\leq \infty$. This is the property of $\Ga$-{\it monotonicity} of
    modulus.
  \item[(c)] For $1<p<\infty$ a unique extremal density $\rho^*$
    always exists. Moreover, in the case of families of walks, there
    always exists an extremal density satisfying $0\leq \rho^*\leq 1$,
    see \cite{abppcw:ecgd2015}.
  \end{itemize}
\end{remark}

\subsection{Connection to classical quantities}

The concept of $p$-modulus generalizes known several classical ways of
measuring the richness of a family of walks~\cite{abppcw:ecgd2015}.
Let $a$ and $b$ be two nodes in $V$ be given. We define the {\it
  connecting family} $\Ga(a,b)$ to be the family of all simple paths
in $G$ that start at $a$ and end at $b$.  To this family, we assign
the usage function $\cN(\gamma,e)$ to be $1$ when $e\in\gamma$ and $0$
otherwise.

\begin{theorem}[\cite{abppcw:ecgd2015}]\label{thm:generalize}
  Let $G=(V,E)$ be a graph. Let $\Ga$ be a
  family of walks on $G$. Then the function
  $ p\mapsto \Mod_{p}(\Ga)$ is continuous for
  $1\leq p< \infty$, and 
  for $1\le p \le p' <\infty$:
  \begin{align}
    \label{eq:monotone-decr}
    \Mod_{p}(\Ga) &\ge \Mod_{p'}(\Ga),\\
    \label{eq:monotone-incr}
    \left(|E|^{-1/p}\Mod_{p}(\Ga)\right)^{1/p} &\le
    \left(|E|^{-1/p'}\Mod_{p'}(\Ga)\right)^{1/p'}.
  \end{align}
  Moreover, let $a\neq b$ in $V$ be given and set $\Ga=\Ga(a,b)$. Then,
  \begin{itemize}
  \item[(i)] {\bf\it For} $\mathbf{p=\infty}$:
    \[
    \lim_{p\to\infty}\Mod_{p}(\Gamma)^{\frac{1}{p}} = \Mod_{\infty}(\Ga)=\frac{1}{\ell(\Ga)}.
    \]
  \item[(ii)] {\bf\it For} ${\mathbf p=1}$, 
    \[
    \Mod_{1}(\Ga)=\min\{|\bd S|:\text{\rm $S$ an $ab$-cut}\} = MC(a,b).
    \]
  \item[(iii)] {\bf\it For} ${\mathbf p=2}$,
    \[
    \Mod_{2}(\Ga)=\cCeff(a,b) = \cReff(a,b)^{-1}.
    \]
  \end{itemize}
\end{theorem}
\begin{remark}
In other words, as $p$ varies continuously from $1$ to $2$ and to $\infty$, the quantity $\Mod_p(\Ga(a,b))$ recovers the classical notions of min cut, effective conductance, and shortest path.
\end{remark}

\begin{example}[Basic Example]\label{ex:basic}
  Let $G$ be a graph consisting of $k$ simple paths in parallel, each
  path taking $\ell$ hops to connect a given vertex $s$ to a given
  vertex $t$. Let $\Ga$ be the family consisting of the $k$ simple
  paths from $s$ to $t$.  Then $\ell(\Ga)=\ell$ and the size of the
  minimum cut is $k$. A straightforward computation shows that
\[
\Mod_p(\Ga)=\frac{k}{\ell^{p-1}}\quad\mbox{for } 1\le p<\infty,\qquad
\Mod_{\infty}(\Ga)=\frac{1}{\ell}.
\]
Intuitively,
when $p\approx 1$, $\Mod_p(\Ga)$ is more sensitive to the number of
parallel paths, while for $p\gg 1$, $\Mod_p(\Ga)$ is more sensitive to
short walks.
\end{example}

\section{The $d_p$ metric}\label{sec:dp}
Reinterpreting Theorem~\ref{thm:generalize} in the context of
Section~\ref{sec:intro}, we see that $\Mod_p(\Ga(a,b))^{-1}$ is a
metric for $p=1,2,\infty$.  One might naturally wonder if this fact
generalizes to all $p\in[1,\infty]$.  The answer turns out to be
``no.''  However, we'll see shortly that introducing a $p$th root does
in fact lead to a metric for all $p$.

\begin{definition}
For $1\le p \le \infty$, let
\[
d_p(a,b):={\rm Mod}_{p}(\Gamma(a,b))^{-1/p}\quad\text{if
}p<\infty\qquad\text{and}\qquad d_\infty(a,b)={\rm
  Mod}_\infty(\Gamma(a,b))^{-1}.
\] 
\end{definition}
Theorem \ref{thm:generalize}(i) implies that
$d_p(a,b)\rightarrow d_\infty(a,b)$ as $p\rightarrow \infty$.
Moreover, the continuity in $p$ and (ii) imply that
$d_p(a,b)\rightarrow d_1(a,b)=\MC(a,b)^{-1}$ as $p\to 1$.
\begin{remark}\label{rem:dtwo}
For $p=2$, $d_2(a,b)=\sqrt{\cReff(a,b)}$. This is a known metric which appears for instance in the context of the discrete Gaussian Free Field, see \cite{ding-lee-peres:annals2012}. It also has the following alternative representation: if $\cG$ is the pseudoinverse of the Laplacian matrix $L$, then 
\[
d_2(a,b)=\sqrt{\cReff(a,b)}=\|\cG^{1/2}\ones_a-\cG^{1/2}\ones_b\|_2,
\]
where $\ones_x$ is the vector with $1$ at $x$ and $0$ everywhere
else. This gives two different ways of verifying that $d_2$ is a
metric. First, given a metric $d$ and an exponent $\ep\in(0,1)$, then
the \emph{snowflaking} $d^\ep$ is always a metric as well (see
Section~\ref{sec:as}). Therefore since, it is known that effective
resistance is a metric, it is immediate that its square-root is a
metric as well. The formulation in terms of $\cG$ shows that $d_2$ is
the pull-back of the Euclidean norm $\|\cdot\|_2$ restricted to the
set of indicator functions $\{\ones_x\}_{x\in V}$ under the linear map
given by $\cG^{1/2}$, again showing that $d_2$ is a metric. Here we take an
alternate approach based on the theory of modulus.
\end{remark}

The main idea in the proof of Theorem \ref{thm:main} below is to compare  the
connecting families $\Ga(a,c)$, $\Ga(c,b)$, $\Ga(a,b)$ and the
\emph{via family} $\Ga(a,b\mid c)$---the family of all walks beginning
at $a$, ending at $b$ and passing through $c$ along the way.  A key
lemma is the following.

\begin{lemma}\label{lem:shpath}
Given a density $\rho:E\rightarrow [0,\infty)$, we have
\[
\ell_{\rho}(\Ga(a,b\mid c))=\ell_{\rho}(\Ga(a,c))+\ell_{\rho}(\Ga(c,b)).
\]
\end{lemma}

\begin{proof}
First, pick $\rho$-shortest walks $\ga_1$ for $\Ga(a,c)$ and $\ga_2$ for $\Ga(c,b)$. Then, 
the concatenation $\ga_0=\ga_1\ga_2$, of $\ga_1$ followed by $\ga_2$, is a walk in $\Ga(a,b\mid c)$. So
\begin{equation}\label{eq:concat}
\ell_{\rho}(\Ga(a,b\mid c))\leq \ell_\rho(\ga_0)=\ell_\rho(\ga_1)+\ell_\rho(\ga_2)=\ell_{\rho}(\Ga(a,c))+\ell_{\rho}(\Ga(c,b)).
\end{equation}
Conversely, let $\ga$ be a walk from $a$ to $b$ via $c$. Write $\ga$ as $\ga' \in \Ga(a,c) $ followed by $\ga'' \in \Ga(c,b) $. Then
\[
\ell_\rho(\ga)=\ell_{\rho}(\ga')+\ell_{\rho}(\ga'')\ge \ell_{\rho}(\Ga(a,c))+\ell_{\rho}(\Ga(c,b)).
\]
Taking the infimum over $\ga\in\Ga(a,b\mid c)$ we get that  $\ell_\rho(\Ga(a,b\mid c)\ge \ell_{\rho}(\Ga(a,c))+\ell_{\rho}(\Ga(c,b))$.
\end{proof}

\begin{theorem}\label{thm:main}
  Let $G=(V,E)$ be a simple connected graph, and let
  $p\in [1,\infty]$. Then, $d_p$ is a metric on $V$.  Moreover, $d_1$
  is an ultrametric.
\end{theorem}
\begin{proof}
  That $d_1$ is an ultrametric is a consequence of
  Theorem~\ref{thm:generalize}(ii) and the fact that the reciprocal of
  minimum cut is an ultrametric, while the fact that $d_\infty$ is a
  metric is a consequence of Theorem~\ref{thm:generalize}(i).

  For $1<p<\infty$, we begin by verifying properties (i)-(iii) in
  Definition \ref{def:metric}. Since modulus is the infimum of a
  non-negative energy, non-negativity holds.  If $a=b$ the connecting
  family $\Ga(a,a)$ contains the constant walk, and then no density
  can be admissible, so the $p$-modulus of $\Ga(a,a)$ is infinity and
  $d_p(a,a)=0$. Conversely, if $a\neq b$, consider the constant
  density $\rho_0\equiv 1$. Then $\ell_0:=\ell_{\rho_0}(\Ga(a,b))$ is
  the shortest-path distance from $a$ to $b$, hence $\ell_0<\infty$
  since $G$ is connected. This implies that the density
  $\rho_1:=\rho_0/\ell_0$ is admissible for $\Ga(a,b)$. Therefore, for
  $p\in [1,\infty)$,
\[
\Mod_p(\Ga(a,b))\leq \cE_p(\rho_1)=\frac{|E|}{\ell_0^p}<\infty,
\]
and $\Mod_\infty(\Ga(a,b))\leq \ell_0^{-1}$, showing that
$d_p(a,b)>0$. Finally, since every path from $a$ to $b$ can be
reversed to a path from $b$ to $a$, it follows that
$\Adm(\Ga(a,b))=\Adm(\Ga(b,a))$, so symmetry holds as well.

It remains to prove the triangle inequality. Without loss of
generality we can assume that $a,b,c\in V$ are distinct.  Let
$\Ga(a,b), \Ga (a,c), \Ga(c,b)$ be the corresponding families of
connecting walks and let $\Ga(a,b\mid c)$ be the family of walks from
$a$ to $b$ via $c$.  Let $\rho^*\in\Adm(\Ga(a,b\mid c))$ be extremal
for $\Mod_p(\Ga(a,b\mid c))$.  Then by Lemma \ref{lem:shpath} and
extremality:
\begin{equation}\label{eq:sumell}
1=\ell_{\rho^*}(\Ga(a,b\mid c)=\ell_{\rho^*}(\Ga (a,c))+\ell_{\rho^*}(\Ga(c,b)). 
\end{equation}

We now consider two possibilities.  First, suppose that
$\ell_{\rho^*}(\Ga (a,c))>0$ and $\ell_{\rho^*}(\Ga(c,b))>0,$ and
define
\[
\rho_1:=\frac{\rho^*}{\ell_{\rho^*}(\Ga(a,c))}\qquad\text{and}\qquad \rho_2:=\frac{\rho^*}{\ell_{\rho^*}(\Ga(c,b))}.
\]
Then $\rho_1\in\Adm(\Ga(a,c))$ and $\rho_2\in\Adm(\Ga(c,b)).$
Writing $\Ga_1:=\Ga(a,c)$, $\Ga_2:=\Ga(c,b)$, and $\Ga:=\Ga(a,b\mid c)$, in order to simplify notation, we get
\begin{align*}
d_p(a,c)+d_p(c,b)& = 
\Mod_p(\Ga_1)^{-1/p}+\Mod_p(\Ga_2)^{-1/p} \\ & \geq \cE_p(\rho_1)^{-1/p}+\cE_p(\rho_2)^{-1/p}\\
& = \cE_p(\rho^*)^{-1/p}\left(\ell_{\rho^*}(\Ga_1)+\ell_{\rho^*}(\Ga_2)  \right)\\ 
&  = \cE_p(\rho^*)^{-1/p} = \Mod_p(\Ga)^{-1/p} \\
\end{align*}
where the second to last equality follows from (\ref{eq:sumell}).

On the other hand, suppose that, say, $\ell_{\rho^*}(\Ga(a,c))=0$.
Then~\eqref{eq:sumell} implies that $\ell_{\rho^*}(\Ga(c,b))=1$, which
implies that $\rho^*\in\Adm(\Ga(c,b))$.  Thus,
\begin{equation*}
  d_p(a,c) + d_p(c,b) \ge d_p(c,b) = \Mod_p(\Ga_2)^{-1/p}
  \ge \cE_p(\rho^*)^{-1/p} = \Mod_p(\Ga)^{-1/p},
\end{equation*}
and similarly if $\ell_{\rho^*}(\Ga(c,b))=0$ and
$\ell_{\rho^*}(\Ga(a,c))=1$.
 
Now we use the $\Ga$-monotonicity of modulus. Note that
$\Ga=\Ga(a,b\mid c) \subset \Ga_0:= \Ga(a,b)$.  Thus,
$\Mod_p(\Ga)\leq \Mod_p(\Ga_0)$, hence
$\Mod_p(\Ga)^{-1/p}\geq \Mod_p(\Ga_0)^{-1/p}=d_p(a,b)$ and the
triangle inequality holds.
\end{proof}

\section{Snowflaking and Antisnowflaking}\label{sec:as}

As we saw in Remark \ref{rem:dtwo}, squaring the metric $d_2$ yields effective resistance, which is known to be a metric on any connected graph. Therefore, we now study the question of finding the largest exponents one can raise each $d_p$ metric to, while maintaining the property of being a metric on arbitrary connected graphs.

Given an arbitrary metric, \emph{snowflaking} provides an interesting
way to generate new metrics on the same set.  This procedure is
described by the following known fact.
\begin{fact}
  Let $d$ be a metric on $X$ and let $0<\epsilon<1$, then
  $d^{\epsilon}$ is also a metric on $X$.
\end{fact}
In other words, raising a metric to a positive fractional power always
results in another metric.  This immediately leads one to ask the
following question.  Given some metric $d$ on $X$, is $d$ the
snowflaked version of some other metric?  In other words, does there
exist a $t>1$ such that $d^t$ is also a metric?  When such a $t$
exists, we shall call the resulting metric $d^t$ an
\textit{antisnowflaking} of $d$.

For finite $X$, the characterization of metrics that can be
antisnowflaked is straightforward.  Suppose $a$, $b$ and $c$ are
distinct points in $X$.  If $d(a,b) < d(a,c) + d(c,b)$, then the
inequality also holds with $d$ replaced by $d^t$ for sufficiently
small $t>1$.  We call such a triple of points $(a,b,c)$ a \emph{proper
  triangle}.  On the other hand, if $d(a,b) = d(a,c) + d(c,b)$, then
it can be seen that $d^t$ violates the triangle inequality for
arbitrarily small $t>1$.  We refer to such a triple as a \emph{flat
  triangle}.  Since a finite set $X$ contains a finite number of
triangles, the following theorem is evident.
\begin{theorem}
  Let $d$ be a metric on a finite set $X$.  There exists a $t>1$ such
  that $d^t$ is a metric on $X$ if and only if $(X,d)$ contains no
  flat triangles.
\end{theorem}
With this in mind, we make the following definition.
\begin{definition}
The {\it antisnowflaking exponent} of a metric $d$ is defined as
\[
\ASFE(d) := \sup \{ t\ge 1: \text{ $d^t$ is a metric} \}
\]
\end{definition}
For instance, it is clear that when $d$ is an ultrametric, then
$\ASFE(d)=\infty$.  While the antisnowflaking exponent of a particular
metric on a particular graph may be interesting in certain contexts,
here we will focus on the best antisnowflaking exponent for an entire
family of connected graphs. Writing $d_p=d_{p,G}$ to show the dependence
on the graph $G$, we define
\begin{equation}\label{eq:antisno}
s(p):=\inf\{\ASFE(d_{p,G}): G\text{ connected}\}.
\end{equation}
Note that if we find a connected graph $G$, an exponent $t\ge 1$, and
three nodes $a,b,c$, such that the triangle inequality for $d_p^t$
fails for this triple, then we are guaranteed that $s(p)\leq t$.  In
particular, by looking at the path graph $P_3$ on three nodes (Figure
\ref{fig:pathgraph}) we can establish the following bound.
\begin{proposition}\label{prop:snowflakeexp}
For $1<p<\infty$:
\begin{equation}\label{eq:snowflakeexp}
s(p)\leq \frac{p}{p-1}=:q
\end{equation}
where $q$ is the H\"{o}lder exponent associated with $p$.

Moreover, the bound is attained for $p=1,2,\infty$:
\[s(1)=\infty,  \qquad s(2)=2, \qquad s(\infty)=1.\]
\end{proposition}

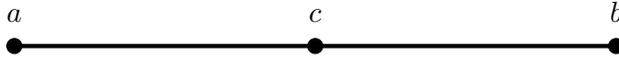
\begin{figure}[h!]
\center
\begin{tikzpicture}[scale=4] 
\draw  [ultra thick]  (0,0) --(1,0);\draw  [ultra thick]  (1,0) --(2,0);

\node [above] at (0,0.05) {\large $a$};\node [above] at (1,0.05) {\large $c$};\node [above] at (2,0.05) {\large $b$};
\draw   [fill] (0,0) circle [radius=0.025];\draw   [fill] (1,0) circle [radius=0.025];\draw   [fill] (2,0) circle [radius=0.025];
\end{tikzpicture}
\caption{The path graph $P_3$ on three nodes.} \label{fig:pathgraph}
\end{figure}

\begin{proof}
  Consider the path graph $P_3$ with nodes $a,c,b$ and fix
  $p\in (1,\infty)$.  It is clear that $d_p(a,c) = 1$, because to be
  admissible a density $\rho$ must satisfy $\rho(a,c)=1$, and then in
  order to minimize the energy, we must also have
  $\rho(c,b)=0$. Likewise, $d_p(c,b) =1$.  For $ d_p(a,b)$, the energy
  is minimized when $\rho(a,c)=\rho(c,b)=1/2$. Thus,
\[
\Mod_p(a,b) = (1/2)^p + (1/2)^p = 2^{1-p}
\]
Hence, $ d_p(a,b) = 2^{(p-1)/p} = 2^{(1-1/p)}$.  The triangle
inequality will fail for $t\ge 1$ such that
\begin{equation}\label{eq:contra}
  d_p(a,b)^t > d_p(a,c)^t + d_p(c,b)^t
\end{equation}
that is,
\[
2^{t(1-1/p)} > 1+1 =2
\]
This happens whenever $t > 1/(1-1/p)$. So $s(p)\le p/(p-1)$.

The bound is attained for the case $p=1$ because, as shown in Theorem
\ref{thm:main}, $d_1$ is an ultrametric on any connected graph, so
$s(1)=\infty$.  When $p=2$, the metric $d_2^2$ is effective
resistance $\cReff$, which is also a metric on connected
graphs. Therefore, $s(2)\ge 2$, attaining the upper bound.  For the
case $p=\infty$, $d_\infty(a,c)=d_\infty(c,b)=1$, while
$d_\infty(a,b)=2$, yielding a flat triangle.  Thus, $s(\infty)=1$.
\end{proof}
In fact, based on the numerical evidence presented in
Section~\ref{sec:num}, we make the following conjecture.
\begin{conjecture}\label{conj:main}
  For all $p\in[1,\infty]$,
  \[
  s(p)=\frac{p}{p-1}
  \]
  Namely, that $P_3$ demonstrates the the worst-case behavior.
\end{conjecture}
In order to prove Conjecture \ref{conj:main}, it is enough to show
that $d_p^q$ is always a metric on connected graphs, thus attaining
the upper bound of Proposition~\ref{prop:snowflakeexp}. We already
know this is true for $p=1,2,\infty$.

\noindent{\bf Added in proof:} Conjecture \ref{conj:main} has been answered in the affirmative, using the theory of Fulkerson duality for modulus, in \cite{acfpc}.

\section{Examples and numerical Results}\label{sec:num}
\subsection{Erd\H{o}s-R\'enyi graphs}
As an attempt to numerically test Conjecture~\ref{conj:main}, we
produced $50$ Erd\H{o}s-R\'enyi graphs on $10$ nodes, with expected
average degree $6$ (discarding any disconnected graphs that were
generated).  For each graph $G_i$, $i=1,2,\ldots,50$, we computed
$d_{p,G_i}(1,2), d_{p,G_i}(2,3),$ and $d_{p,G_i}(1,3)$ for a range of
$p$ values and determined the value $t_{p,i}$ such that
\begin{equation*}
  d^{t_{p,i}}_{p,G_i}(1,2) = d^{t_{p,i}}_{p,G_i}(1,3) + d^{t_{p,i}}_{p,G_i}(2,3).
\end{equation*}
We then estimated $s(p)$ as
\begin{equation*}
  s(p) \le t(p) = \min_{i=1,2,\ldots,50}t_{p,i}.
\end{equation*}
The resulting bound is shown in Figure \ref{fig:erdosreniy} in
blue. The red line in the same figure is the conjectured
antisnowflaking exponent $s(p)=\frac{p}{p-1}$.

\begin{figure}[h!]
\centering
  \includegraphics[trim={0cm 7cm 0cm 7cm},clip, width= 0.9\linewidth]{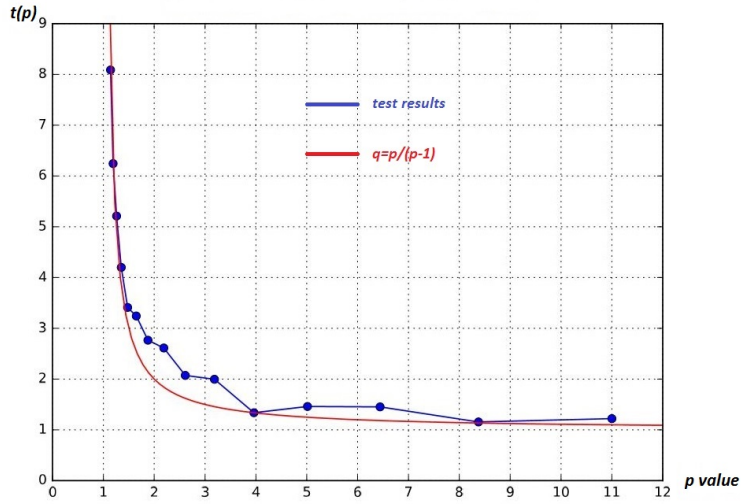}
\caption{Antisnowflaking exponent for different $p$ values.}\label{fig:erdosreniy}
  \end{figure}

Observe that the blue line never goes below the red line. If the blue line had dipped under the red line, that would have been a counter-example for Conjecture \ref{conj:main}. In other words, the worst case scenario seems to be $s(p)=\frac{p}{p-1}$.

In the following we compute some specific examples. When calculating modulus, we will often just write down the extremal metric. For simple examples, verifying that a metric $\rho$ is extremal for $p$-modulus can be done using Beurling's criterion. We state the criterion here for the reader's convenience. For a proof, see \cite[Theorem 2.1]{apc}. 
\begin{theorem}[Beurling's Criterion for Extremality]
\label{thm:beurling}
Let $G$ be a simple graph, $\Ga$ a family of walks on $G$, and $1<p<\infty$.  
Then, a density $\rho\in \Adm(\Ga)$ is extremal for $\Mod_p(\Ga)$, if there is a subfamily $\tilde{\Ga}\subset\Ga$ with $\ell_\rho(\ga)=1$ for all $\ga\in \tilde{\Ga}$, such that for all $h\in\R^E$:
\begin{equation}
\label{eq:beurling}
\mbox{$\sum_{e\in E}\cN(\ga,e)h(e)\geq 0$,\quad for all $\ga\in\tilde{\Ga}$}\quad\Longrightarrow\quad\sum_{e\in E}h(e)\rho^{p-1}(e)\geq 0.
\end{equation}
\end{theorem}

\subsection{Biconnected graphs}\label{sec:biconnected}
Next, we explore the anti-snowflaking exponent for more restrictive families of graphs.
If $\cG$ is a family of connected graphs, define $s_\cG(p):=\inf_{G\in \cG}ASFE(d_{p,G})$.
Clearly $s(p)\le s_{\cG}(p)$. Recall that $s(p)$ is an infimum over all connected graphs. What happens if we restrict the infimum to biconnected graphs?
\begin{definition}
A {\it biconnected graph} is a graph that remains connected after removing any node.
\end{definition}
Let $\cB$ be the family of all biconnected simple graphs. If Conjecture \ref{conj:main} holds, then
from the proof of Proposition \ref{prop:snowflakeexp}, we see that the family of path graphs $\cP$ satisfies
$s(p)=s_{\cP}(p)$ for all $p's$. However, path graphs are not biconnected. So it is natural to wonder if $s(p)<s_{\cB}(p)$. If Conjecture \ref{conj:main} holds, the answer is no.  We establish this by looking at the simplest example of a biconnected graph, namely the cycle graph $C_N$. For distinct nodes $a$, $b$ and $c$, as in Figure \ref{fig:cycle}, consider the connecting families $\Ga(a,b),\Ga(a,c),\Ga(c,b)$.

\begin{figure}[h!]
\begin{minipage}{.45\linewidth}
\begin{flushleft}
\begin{tikzpicture}[scale=0.45] 
\draw   [fill] (0,0) circle [radius=0.1];

\draw   [fill] (2.5,6) circle [radius=0.2];
\draw   [fill] (0,6.5) circle [radius=0.2];
\draw   [fill] (-5.7,3.2) circle [radius=0.2];
\draw   [fill] (-4.25,4.9) circle [radius=0.2];
\draw   [fill] (-6.40,1.1) circle [radius=0.2];
\draw   [fill] (-2.3,6.1) circle [radius=0.2];
\draw   [fill] (-6.45,-1.1) circle [radius=0.2];

\node   [above] at (1.5,6.4)  {\small $\frac{1}{N-1}$};
\node  [above]  at (-1.4,6.5) {\small $\frac{1}{N-1}$};
\node  [right] at (-5.5,3) {\large  a};
\node  [right] at (-4.25,4.5) {\large  c};
\node [right] at (-2.3,5.7) {\large b};
\node  [above] at (-7.3,1.3) {\small $\frac{1}{N-1}$};
\node  [above] at (-7.4,-1.0) {\small $\frac{1}{N-1}$};
\node  [above] at (-5.5,3.8) {\small 1};
\node  [above] at (-4.25,5.3) {\small $\frac{1}{N-1}$};

\draw [thick] (2.5,6) arc [radius=6.5, start angle=67, end angle=190]; 
\draw [dashed, thick] (2.5,6) arc [radius=6.5, start angle=67, end angle=-170];
\end{tikzpicture}
\end{flushleft}
\end{minipage}
\hfill
\hspace{1.0cm}
\begin{minipage}{.45\linewidth}
\begin{flushright}
\begin{tikzpicture}[scale=0.45] 
\draw   [fill] (0,0) circle [radius=0.1];

\draw   [fill] (2.5,6) circle [radius=0.2];
\draw   [fill] (0,6.5) circle [radius=0.2];
\draw   [fill] (-5.7,3.2) circle [radius=0.2];
\draw   [fill] (-4.25,4.9) circle [radius=0.2];
\draw   [fill] (-6.4,1.1) circle [radius=0.2];
\draw   [fill] (-2.3,6.1) circle [radius=0.2];
\draw   [fill] (-6.45,-1.1) circle [radius=0.2];

\node   [above] at (1.5,6.4)  {\small $\frac{1}{N-2}$};
\node  [above]  at (-1.4,6.5) {\small $\frac{1}{N-2}$};
\node  [right] at (-5.5,3) {\large  a};
\node  [right] at (-4.25,4.5) {\large  c};
\node [right] at (-2.3,5.7) {\large b};
\node  [above] at (-7.10,1.5) {\small $\frac{1}{N-2}$};
\node  [above] at (-7.3,-1.0) {\small $\frac{1}{N-2}$};
\node  [above] at (-5.6,4.0) {\small $\frac{1}{2}$};
\node  [above] at (-4,5.4) {\small $\frac{1}{2}$};

\draw [thick] (2.5,6) arc [radius=6.5, start angle=67, end angle=190]; 
\draw [dashed, thick] (2.5,6) arc [radius=6.5, start angle=67, end angle=-170];
\end{tikzpicture}
\end{flushright}
\end{minipage}
\caption{The cycle graph $C_N$ and the extremal density $\rho^*$ for $\Ga(a,c)$ and $\Ga(a,b)$.}\label{fig:cycle}
\end{figure}
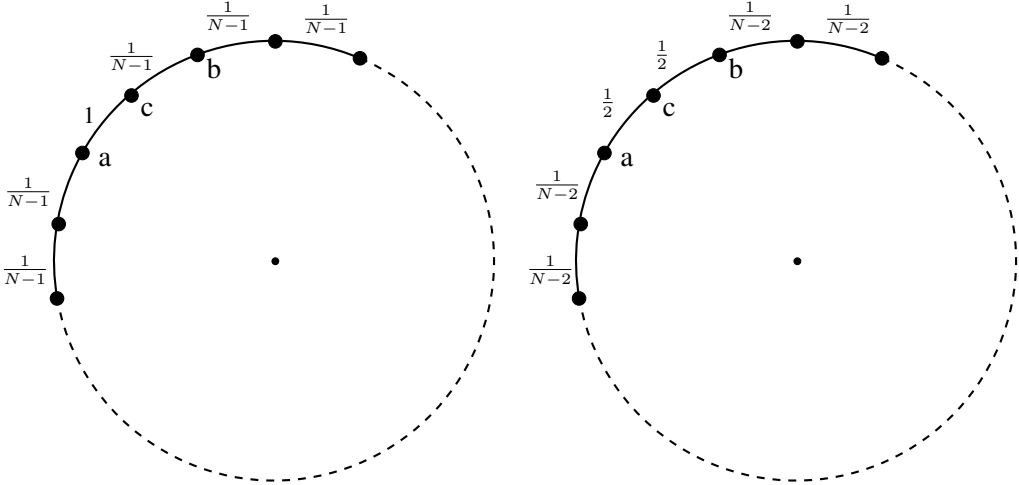

The middle diagram in Figure \ref{fig:cycle}  shows the extremal density $\rho^*$ for the family $\Ga(a,c)$.
The subfamily $\tilde{\Ga}$ that verifies Beurling's criterion in this case has two simple paths, namely $a\ c$ and the path from $a$ to $c$ that traverses the cycle in the other direction (the long way).
This gives
\begin{equation}\label{eq:modac}
\Mod_p(\Ga(a,c)) =1+\left(\frac{1}{N-1}\right)^p (N-1)
=  1 + (N-1)^{1-p}.
\end{equation}
By symmetry,  $\Mod_p(\Ga(c,b)) = 1 + (N-1)^{1-p}$ as well.
We conclude that 
\[
d_p(a,c)=d_p(c,b) =  (1 + (N-1)^{1-p})^{-1/p}.
\]
The right-most diagram in  Figure \ref{fig:cycle}   shows the extremal density $\rho^*$ for the family $\Ga(a,b)$. The subfamily $\tilde{\Ga}$ that verifies Beurling's criterion in this case has two simple paths, namely $a\ c\ b$ and the longer path from $a$ to $b$ in the other direction. As a consequence,
\begin{equation}\label{eq:modab}
\Mod_p(\Ga(a,b)) =2\left(\frac{1}{2}\right)^p +(N-2)\left(\frac{1}{N-2}\right)^p 
=  2^{1-p} + (N-2)^{1-p}.
\end{equation}
Therefore, $d_p(a,b) = (2^{1-p} + (N-2)^{1-p})^{-1/p}$.
We will calculate the infimum of all the exponents  $t \ge 1$ for which the following triangle inequality fails:
\[
d_p(a,b)^t > d_p(a,c)^t+d_p(c,b)^t.
\]
We get
\[
\left(2^{1-p} + (N-2)^{1-p}\right)^{-t/p} > 2(1 + (N-1)^{1-p})^{-t/p},
\]
hence
\[
 2^{p/t}<\frac{1 + (N-1)^{1-p}}{2^{1-p} + (N-2)^{1-p}}.
\]
So the  infimal exponent is
\begin{equation*}\label{eq:to}
t_0 := p\left[\log_2\left(\frac{1 + (N-1)^{1-p}}{2^{1-p} + (N-2)^{1-p}}\right)\right]^{-1}.
\end{equation*}

We see that as $ N \rightarrow \infty $, $t_0  \rightarrow \frac{p}{-(1-p)}= \frac{p}{(p-1)} = q $.
Therefore, if we let $\cB$ denote the family of biconnected graphs and define $s_\cB(p):=\inf_{G\in \cB}ASFE(d_{p,G})$, then we see that $s(p)\le s_{\cB}(p)\le p/(p-1)$, with equality in both places if Conjecture \ref{conj:main} holds.

\subsection{Complete graphs}
The {\it complete graph} $K_N$ is a simple graph on $N$ nodes, where every node is connected to each other, see Figure \ref{fig:complete}.

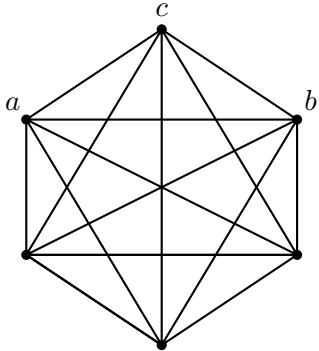
\begin{figure}[h!]
\center
\begin{tikzpicture}[scale = 0.6]
\draw  [thick]  (3,0) --(6,2);
\draw  [thick]  (6,2) --(6,5); 
\draw  [thick]  (6,5) --(3,7);
\draw  [thick]  (3,7) --(0,5);
\draw  [thick]  (0,2) --(0,5); 
\draw  [thick]  (0,2) --(3,0);
\draw   [thick] (0,2) --(6,2);\draw  [thick]  (0,2) --(6,5);\draw  [thick]  (0,2) --(3,7);\draw  [thick]  (3,0) --(0,5);\draw  [thick]  (3,0) --(3,7);
\draw  [thick]  (0,2) --(3,0);\draw  [thick]  (6,2) --(0,5);\draw  [thick]  (6,2) --(3,7);\draw  [thick]  (6,2) --(0,2);\draw  [thick]  (3,0) --(6,5);
\draw  [thick]  (0,5) --(6,5);
\draw   [fill] (3,0) circle [radius=0.1];\draw [fill] (6,2) circle [radius=0.1];\draw [fill] (6,5) circle [radius=0.1];\draw [fill] (3,7) circle [radius=0.1];\draw [fill] (0,5) circle [radius=0.1];\draw [fill] (0,2) circle [radius=0.1];
\node [above] at (-0.3,5) {\large $a$};
\node [above] at (6.3,5) {\large $b$};
\node [above] at (3,7.1) {\large $c$};
\end{tikzpicture}
\caption{$K_6$- Complete graph on 6 nodes.}\label{fig:complete}
\end{figure}

 Observe that, by symmetry, $\Mod_p(\Ga(a,c))=\Mod_p(\Ga(c,b)=\Mod_p(\Ga(a,b))$, hence $d_p(a,c)=d_p(c,b)=d_p(a,b)$.
Therefore, $d_p$ is an ultrametric on complete graphs. In particular, if $\cK$ is the family of complete graphs, then
$s_\cK(p)=\infty$ for all $p$.

It's still interesting to compute $d_p(a,b)$ for an arbitrary pair of nodes.
Figure \ref{fig:square} depicts the extremal density $\rho^*$ for  $\Ga(a,b)$ in  $K_N$.

\begin{figure}[h!]
\center
\begin{tikzpicture}[scale=0.55] 
\draw  [dashed, ultra thick]  (-6.5,0) --(2.5,6); \draw  [ultra thick]  (-2.5,6) --(2.5,6);
\draw  [ultra thick]  (-5.10,4) --(-2.5,6);\draw  [ultra thick]  (5.10,4) --(2.5,6);
\draw  [dashed, ultra thick]  (6.5,0) --(-2.5,6);\draw  [ultra thick]  (-5.10,4) --(2.5,6);\draw  [ultra thick]  (5.10,4) --(-2.5,6);\draw  [dashed, ultra thick]  (-6.5,0) --(-2.5,6); \draw  [dashed, ultra thick]  (6.5,0) --(2.5,6);
\draw  [ultra thick]  (-4.15,-5) --(1.5,-6.3);
\draw  [ultra thick]  (-2.5,-6) --(2.5,-6);

\draw   [fill] (0,0) circle [radius=0.1];
\draw   [dashed, thick] (0,0) circle [radius=6.5];

\node [above] at (2.5,6.1) {\large $b$};
\node [above] at (-2.5,6.1) {\large $a$};

\draw   [fill] (2.5,6) circle [radius=0.1];\draw   [fill] (-2.5,6) circle [radius=0.1];\draw   [fill] (5.10,4) circle [radius=0.1];\draw   [fill] (-5.10,4) circle [radius=0.1];
\draw   [fill] (-6.5,0) circle [radius=0.1];\draw   [fill] (6.5,0) circle [radius=0.1];\draw   [fill] (-4.15,-5) circle [radius=0.1];\draw   [fill] (1.5,-6.3) circle [radius=0.1];\draw   [fill] (-2.5,-6) circle [radius=0.1];\draw   [fill] (2.5,-6) circle [radius=0.1];

\node [above] at (0,6) { 1};
\node [above] at (3.9,5.1) { 0.5};
\node [above] at (-3.9,5.1) { 0.5};
\node [below] at (-2.3,5.6) { 0.5};
\node [below] at (2.3,5.6) {0.5};
\node [above] at (-5.2,2.5) { 0.5};
\node [above] at (5.2,2.5) {0.5};
\node [above] at (2.9,2.5) {0.5};
\node [above] at (-2.7,2.5) { 0.5};
\node [above] at (-2.1,-5.5) { 0};\node [above] at (1.8,-6) { 0};
\end{tikzpicture}
\caption{The complete graph $K_N$ and the extremal density $\rho^*$ for $\Ga(a,b)$.} 
\end{figure}
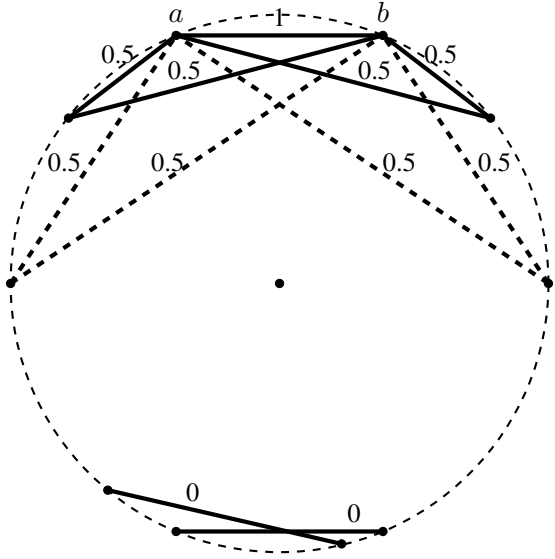

In formulas, $\rho^*(a,x)=1/2=\rho^*(b,x)$ for every $x\neq a,b$, and $\rho^*(a,b)=1$, otherwise $\rho^*$ is zero. To verify Beurling's criterion, consider the subfamily $\tilde{\Ga}$ of simple paths consisting of $a\ b$ and $a\ x\ b$ for any $x\neq a,b$.
We get that
\[
\Mod_p(\Ga(a,b)) = 1+ 2 (N-2) \frac{1}{2^p}\qquad\text{and}\qquad d_p(a,b)=\left(1+\frac{N-2}{2^{p-1}}\right)^{-1/p}.
\]

Since $s(p)\le s_\cB(p)\le p/(p-1)$ while $s_\cK(p)=\infty$, what are some natural families of graphs for which $p/(p-1)<s_{\cG}(p)<\infty$? 
For instance, what happens for the family of all hypercubes? Recall that
for an integer $N\geq 2$, the {\it hypercube} $H_N$ is the graph whose nodes are strings of $0$ and $1$ of length $N$ and two such strings are connected by an edge if they differ in exactly one position.

\subsection{Graph visualization}\label{sec:visual}

Metrics on networks play a vital role in applications as well as in
the study of intrinsic network characteristics.  For instance, there
are infinitely many ways to draw a network in two- or
three-dimensional space.  However, some choices of node layout are
clearly better than others for providing a meaningful visualization of
the network. Take a cycle graph on $5$ nodes, for example.  Drawing a
regular pentagon provides a much better representation of this graph
than does placing the $5$ nodes randomly in the plane. To relate this
to metrics, one need only observe that any time we draw a graph in the
plane, its node set inherits the Euclidean metric of the plane.  In
this sense, different drawings of the same graph $G$ represent
different choices of metric on the vertices $V$ and it thus seems
natural that the choice of layout should be closely related to the
network structure. For a beautiful example of deriving a network's
layout from its intrinsic structure,
see~\cite[Sec.~2.2]{spielman:bams2017}. 
Here, we briefly discuss the relationship between graph visualization and the metrics that we are studying in this paper.

A mapping $f: (X, d_X ) \rightarrow (Y, d_Y )$ of one metric space into another is called
an {\it isometric embedding} or {\it isometry} if $d_Y (f(x), f(y)) = d_X (x, y)$ for all
$x, y \in X$. Two metric spaces are isometric if there exists a bijective isometry between
them.
\begin{theorem}[Shoenberg, 1935]
Given a finite metric space $X= \{x_0, ...,x_m\}$ and an integer $n\in\bN$, $X$ embeds isometrically into $\R^n$ if and only if the matrix $M \in \R^{m\times m}$ whose entries are 
\[
d(x_i,x_0)^2 + d(x_0,x_j)^2-d(x_i,x_j)^2
\] 
is positive semi-definite and of rank less than or equal to $n $.
\end{theorem}
For convenience, we will call the matrix $M$ the ``Shoenberg matrix''.
As an example, consider the square in Figure \ref{fig:square}.
We fix the node $a$ and derive the Shoenberg matrix so we can analyze the embeddings of this graph when the nodes are endowed with modulus metrics of the form $d_p^t$ for some $t>0$.
By symmetry, these metrics have the property that the distance between two neighboring nodes of the square is a constant $\alpha>0$ and the distance between diagonally opposite nodes is some other constant $\beta>0$.

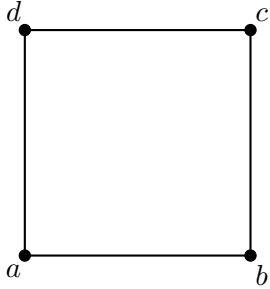
\begin{figure}[h!]
\center
\begin{tikzpicture}[scale=3] 
\draw  [thick]  (0,0) --(0,1);\draw  [thick]  (0,1) --(1,1);\draw  [thick]  (1,1) --(1,0);\draw  [thick]  (0,0) --(1,0);
\node [below] at (-0.05,0) {\large $a$};\node [below] at (1.05,0) {\large $b$};\node [above] at (-0.05,1) {\large $d$};\node [above] at (1.05,1) {\large $c$};
\draw   [fill] (0,0) circle [radius=0.025];\draw   [fill] (1,0) circle [radius=0.025];\draw   [fill] (0,1) circle [radius=0.025];\draw   [fill] (1,1) circle [radius=0.025];
\end{tikzpicture}
\caption{The square graph}\label{fig:square}
\end{figure}

Let the columns and rows of the matrix $M$ represent nodes $b, c, d$ respectively. The entry $M_{11}$ represents $(b,b)$ and is calculated as follows:
\[
d({\bf a},b)^2+d({\bf a},b)^2 - d(b,b)^2 =2\alpha.
\]
Note that this will be the case for $(d,d)$ as well. On the other hand, for $(c,c)$:
\[
d({\bf a},c)^2+d({\bf a},c)^2 - d(c,c)^2 = 2\beta^2
\]
The entry $M_{12}$ represents $(b,c)$ and is calculated as follows:
\[
d({\bf a},b)^2+d({\bf a},c)^2 - d(b,c)^2 = \beta^2
\]
 Likewise for $(c,d)$ we get $\beta^2$ again. For $(b,d)$ we have
\[
d({\bf a},b)^2+d({\bf a},d)^2 - d(b,d)^2 = 2\alpha^2-\beta^2
\]
 Putting the above information together, we can derive a Shoenberg matrix $M$ for these type of metrics on the square.
\[
M=
\begin{bmatrix}
    2\alpha      & \beta^2 & 2\alpha^2-\beta^2   \\
    \beta^2      & 2\beta^2 &  \beta^2  \\
    2\alpha^2-\beta^2      & \beta^2 & 2\alpha
\end{bmatrix}
\]
Note that for the triangle inequality to hold, we also want the condition, 
\[
\beta\le 2\alpha.
\]
Without loss of generality, we can normalize the edge distance to be $1$, setting $\alpha=1$, and then plot how the eigenvalues of $M$ change with $\beta$. In Figure \ref{fig:eigenvalues}, we have plotted the eigenvalues of $M$ as the normalized parameter which we still call $\beta$ varies from $0$ to $2$.
\begin{figure}[h!]
			\centering
			\includegraphics[width= 0.6\linewidth]{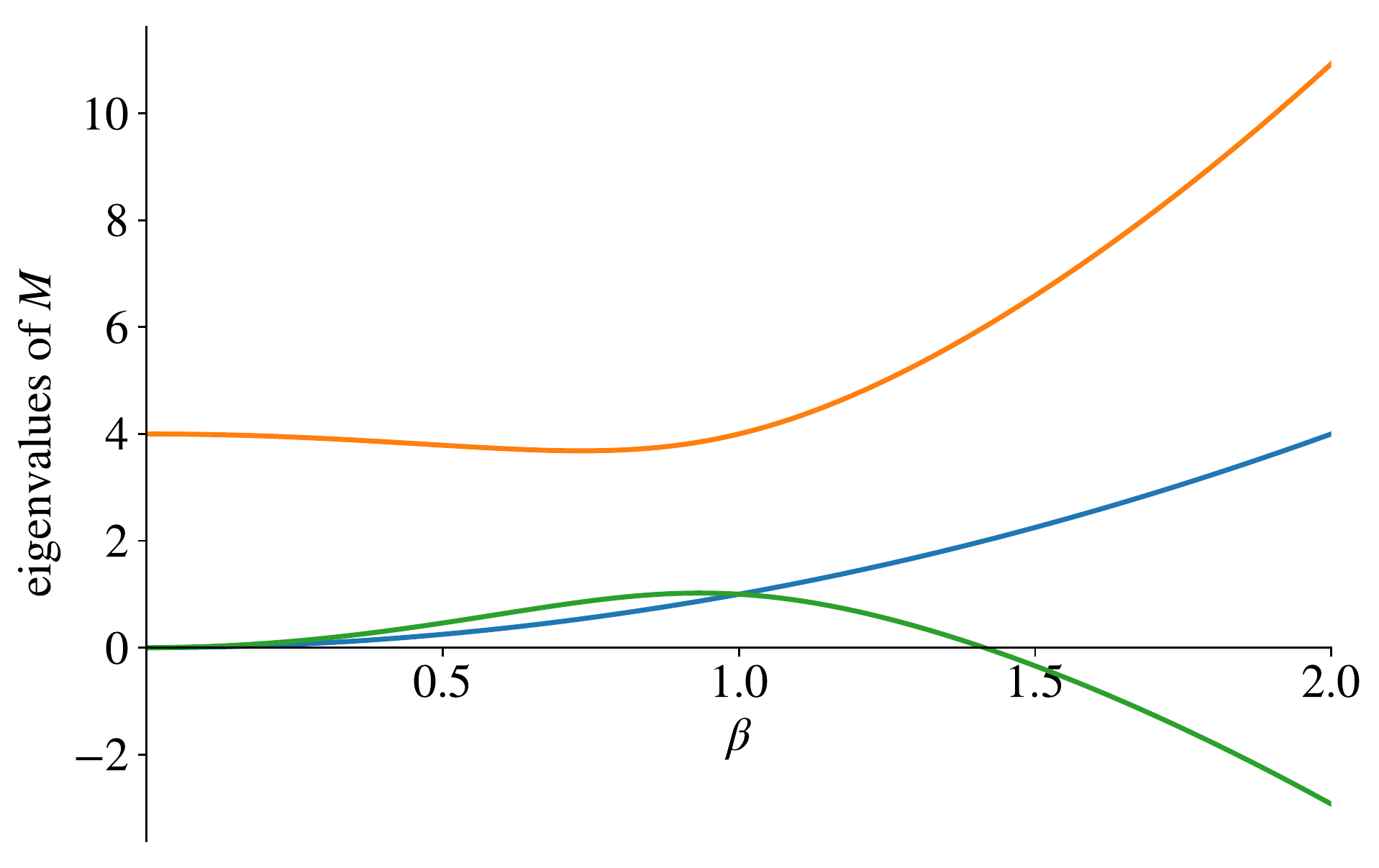}
\caption{Eigenvalues of $M$ as $\beta$ varies, given $\alpha=1$.}\label{fig:eigenvalues}
		\end{figure}
We observe that when $\beta>\sqrt{2} \approx 1.4$, the matrix $M$ starts having negative eigenvalues and thus fails to be positive semi-definite.
For the $\beta$ in $(0,\sqrt{2})$ the square is embeddable in $\R^3$ and for $\beta=\sqrt{2}$ it is embeddable in $\R^2$.  We describe these embeddings by fixing one edge of the square at $(1/2, 0, 0)$ and $(-1/2, 0, 0)$, while the opposite edge is horizontal at some height $h>0$ and by symmetry is also centered on the vertical axis. Since all four edges have unit length, the top horizontal edge must twist about the vertical axis  by an an angle $\theta\ge 0$.
The relationship between $h$ and $\theta$ is governed by the parameter $\beta$, the distance between diagonally opposite nodes. A simple calculation shows that
\[
h=\cos\frac{\theta}{2}\qquad\text{and}\qquad \beta=\sqrt{1+\cos\theta}.
\]
When $\theta=0$, $\beta=\sqrt{2}$ and the square is in the $xz$-plane. But for $\theta\in(0,\pi)$, the embedding is three-dimensional and $h$ tends to $0$ as $\theta$ tends to $\pi$. When $\theta=\pi$, $\beta=0$ and $d$ stops being a metric.

To connect this to our modulus metrics, note that the square is a special case of the cycle graphs we discussed in Section \ref{sec:biconnected}, when $N=4$. So by (\ref{eq:modac}) and (\ref{eq:modab}) applied to the edge $(a,b)$ and the diagonal $(a,c)$ in this case, we see that 
\[
d_p(a,b)=(1+3^{1-p})^{-1/p}\qquad\text{and}\qquad d_p(a,c)=2^{1-2/p}.
\]
Therefore, the ratio $\beta/\alpha$ for the case of the $d_p$ metrics is
\[
f(p):=2^{1-2/p}(1+3^{1-p})^{1/p}.
\]
A computation shows that $f$ is increasing, thus as $p$ decreases to $1$ the ratio $\beta/\alpha$ decreases to $f(1)=1$. Hence, the metric $d_p$ on the square graph is embeddable in $\R^3$ for $1\le p\le p_0$, for some value $p_0\approx 3.88$.

\subsection{Numerical algorithms}

There are a number of options for computing the modulus metrics numerically.  Here, we provide a short overview of some of the most efficient methods known to date.

\subsubsection{The direct approach}

One method for computing $d_p$ is to use the formulation of modulus given in Definition~\ref{def:mod} directly.  This formulates modulus as a convex optimization problem with finitely many affine inequality constraints.  For small graphs, one could simply enumerate all simple paths from $a$ to $b$, form the usage matrix, and pass the problem to any of several standard convex optimization solvers.

The problem with this approach, of course, is that the number of constraints tends to grow combinatorially with the graph size, making it computationally infeasible to even enumerate all constraints for larger graphs.  A modification that has proven effective in practice is the greedy algorithm described in~\cite{ASGPC}.  The idea is to iteratively build a subfamily $\Gamma'\subset\Gamma$ in such a way that $|\Gamma'|\ll|\Gamma|$ and $\Mod_p(\Gamma')\approx \Mod_p(\Gamma)$.  Initially, $\Gamma'$ is the empty set.  On each iteration of the algorithm, the smaller modulus problem $\Mod_p(\Gamma')$ is solved to obtain the optimal density $\rho'$.  If this density is admissible for the full problem, then the $\Gamma$-monotonicity property of modulus shows that it is optimal and $\Mod_p(\Gamma')=\Mod_p(\Gamma)$.  If it is not admissible, then more paths should be added to $\Gamma'$.  The ``greedy'' implementation of this algorithm is to add the path for which $\ell_{\rho'}(\gamma)=\ell_{\rho'}(\Gamma)$ (i.e., the ``most violated constraint'').  This approach can be modified with the stopping condition $\ell_{\rho'}(\Gamma) \ge 1-\epsilon_{\text{tol}}$, which gives an approximation to the modulus with both upper and lower bounds given in terms of the tolerance $\epsilon_{\text{tol}}$.

\subsubsection{The potential formulation}

The direct method described above is applicable to computing the modulus not only of connecting families, but also of more general families of objects on $G$. If we restrict attention to connecting families alone, there is an equivalent formulation in terms of vertex potentials~\cite[Theorem 4.2]{abppcw:ecgd2015}.  For $1\le p<\infty$, the modulus can be rewritten as a minimization over vertex potentials $\phi:V\to\mathbb{R}$ as follows.  For $a\ne b$ in $V$, we solve the problem
\begin{align*}
    \text{minimize}\quad& \sum_{\{x,y\}\in E}|\phi(x)-\phi(y)|^p\\
    \text{subject to}\quad& \phi(a)=0,\quad\phi(b)=1.
\end{align*}

The value of this problem is exactly $\Mod_p(\Gamma(a,b))$.  Moreover, the optimal density $\rho^*$ can be recovered from the optimal $\phi^*$ as $\rho^*(x,y)=|\phi^*(x)-\phi^*(y)|$, for every edge $\{x,y\}\in E$.  This provides a smaller (both in terms of unknowns and constraints) convex optimization formulation for modulus.

\subsubsection{Special cases}

There are also a three special cases, as can be seen in Theorem~\ref{thm:generalize}, for which the modulus metric can be computed using other known algorithms.  Since the $p=\infty$ case is equivalent to computing the graph distance between $a$ and $b$, a simple breadth-first search algorithm can be used to compute $d_\infty$.  Similarly, $d_1$ can be computed using any min-cut algorithm, and $d_2$ can be computed using any algorithm for computing effective resistance.

For $p=2$, moreover, there is also an efficient approach that can be used if all pairwise distances are needed.  (It is not clear if similar methods exist for other values of $p$, though it seems plausible that knowledge of some pairwise distances could accelerate the computations of others.)  For $p=2$, the $d_p$ distance is closely related to the graph Laplacian operator $L$.  In fact, it is known that
\begin{equation*}
  \mathcal{R}_{\text{eff}}(a,b) = (\delta_a-\delta_b)L^+(\delta_a-\delta_b)
  = L^+(a,a)+L^+(b,b)-2L^+(a,b),
\end{equation*}
where $L^+$ is the Moore-Penrose pseudoinverse of $L$.  Thus, any method for efficiently computing $L^+$ leads to an efficient method for computing pairwise $d_2$ distances.

\subsubsection{Numerical comparisons}

\begin{figure}
  \centering
  \includegraphics[width=0.6\textwidth]{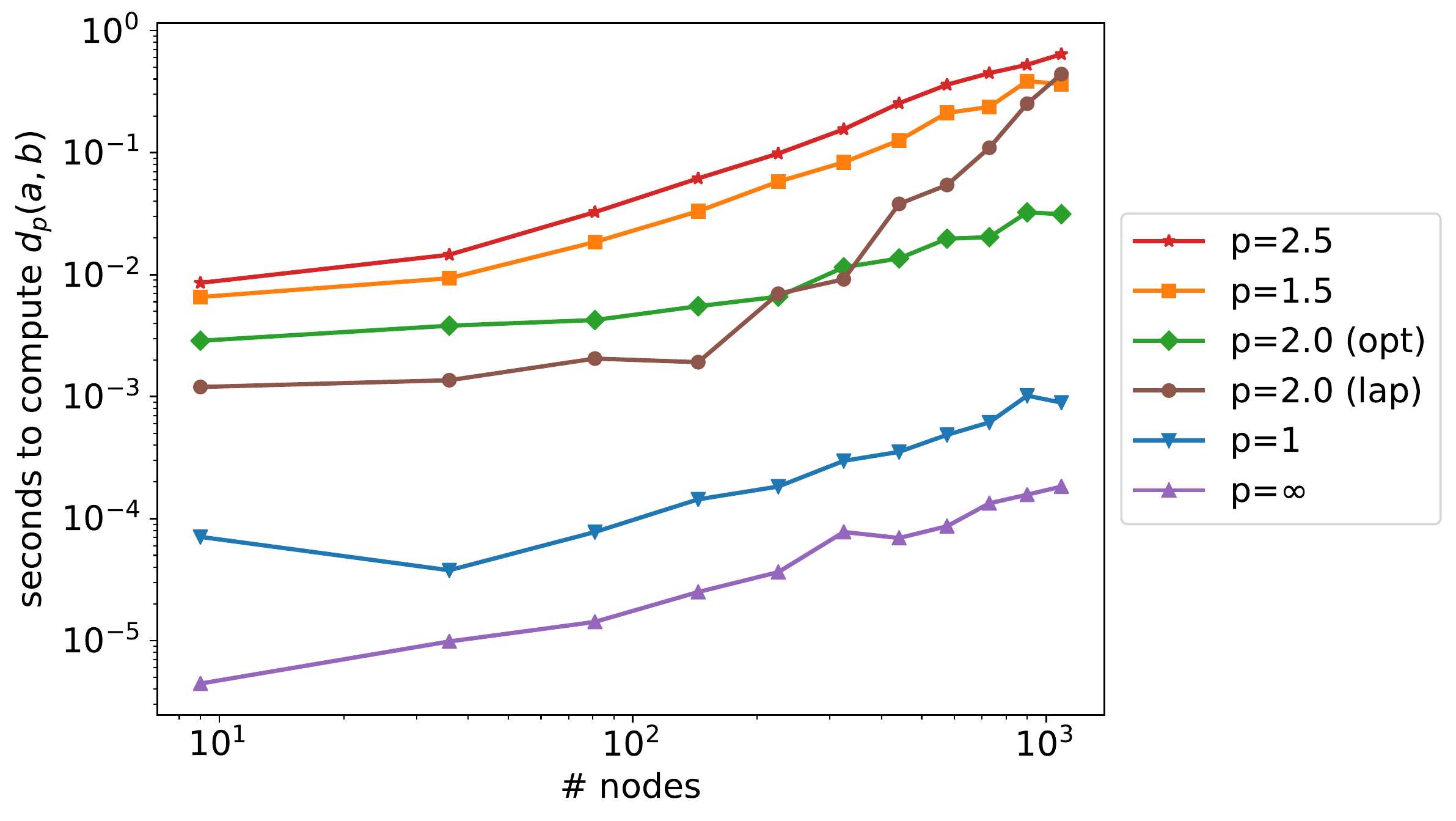}
  \caption{Comparisons of times required to compute $d_p$ distances on several square 2D grids for different values of $p$.}
  \label{fig:compute-times}
\end{figure}

A comparison of the time required to compute $d_p$ for various $p$ is shown in Figure~\ref{fig:compute-times}.  For this comparison, we measured the time required to compute the distance between opposite corners of an $n\times n$ square grid (with nodes connected to their horizontal and vertical neighbors) for $n=3,6,9,12,\ldots,33$.  All tests were performed on a 2.7GHz laptop computer using Python code.  The cases $p=1.5$ and $p=2.5$ were computed using the vertex potential formulation via the \texttt{cvxpy} package~\cite{cvxpy}.  The cases $p=1$ and $p=\infty$ were solved respectively using the minimum cut and shortest path algorithms of the \texttt{igraph} package~\cite{igraph}.  The case $p=2$ was computed in two different ways.  For the curve labeled ``2.0 (opt),'' the potential formulation was solved (as a quadratic program) by \texttt{cvxpy}.  For the curve labeled ``2.0 (lap),'' the pseudoinverse of the graph Laplacian was computed using \texttt{scipy.sparse.linalg}~\cite{scipy}.

For all cases other than ``2.0 (lap),'' we averaged the time to compute $d_p(a,b)$ over 20 trials to help filter the noise incurred by computing in a multi-tasking environment.  For the Laplacian case, we computed the time required to find $d_2(a,b)$. However, the comparison is complicated by the fact that once the pseudoinverse of the Laplacian  is computed this method will quickly deliver all other pairwise distances.

\section{Future research}

\noindent $\bullet$ We hope to be able to prove that $d_p^q$ is always a metric as conjectured in Conjecture \ref{conj:main}. The plan is to try and generalize one of the known proofs in the $p=2$ case.

\noindent{\bf Added in Proof:} The proof we gave in \cite{acfpc} is not a generalization of known proofs, rather it gives a new proof even in the $p=2$ case.

\noindent $\bullet$ Also, we showed in ~\cite{GASSP} that
\[
d_2^2(a,b)=\cReff(a,b)\le {\rm EHT}(a,b) \le d_\infty(a,b)
\]
where $\cReff(a,b)$ is the effective resistance and ${\rm EHT}(a,b)$ is epidemic hitting time. Is there  a relationship between the epidemic hitting time metric and  $d_p^q$ for some $p$?

\noindent $\bullet$ The example of the square graphs in Section \ref{sec:visual} raises the following question: Is it true that for an arbitrary simple graph $G=(V,E)$, there is always a value  $p_0$ so that $(V,d_{p_0})$ embeds isometrically in some $\R^n$, and if so, does this imply the same is true for all $p<p_0$? Similar questions can be asked for other powers of $d_p$.

\noindent $\bullet$ In the example of the square graph in Section \ref{sec:visual}, the $d_p$ metric can be raised to an exponent that is strictly larger than $q$. In fact, to find the largest possible exponent $t_{\rm max}$ in this case it's enough to set
\[
d_p(a,c)^t=2 d_p(a,b)^t,
\]
and solve for $t$. When this happen we say that the graph contains a ``flat triangle''.
A calculation shows that 
\[
t_{\rm max}=\frac{1}{1-\frac{1}{p}\left(2-\log_2(1+3^{1-p})\right)}>\frac{1}{1-\frac{1}{p}}=q.
\]
More generally, we intend to do the same computation for the cube in $\R^3$ and study whether
the anti-snowflaking exponent $s_{\cH}(p)$ can be computed for the family of hypercubes $\cH$.
These are graphs that arise in the theory of expander graphs and are considered to be very well connected, hence, their triangles should be far from being flat.

\noindent $\bullet$ Finally, we are interested in studying the monotonicity properties of the metric $d_p$ and the conjectured metric $d_p^q$.

\section*{Acknowledgment}
The authors thank the anonymous referees for helpful comments that have improved the paper.


\providecommand{\href}[2]{#2}
\providecommand{\arxiv}[1]{\href{http://arxiv.org/abs/#1}{arXiv:#1}}
\providecommand{\url}[1]{\texttt{#1}}
\providecommand{\urlprefix}{URL }

\medskip
\medskip

\end{document}